\documentclass[11pt]{amsart}
\usepackage{amssymb,amsmath,amsthm,amsfonts,verbatim}
\usepackage[all,cmtip]{xy}
\input diagxy
\usepackage{hyperref}
\usepackage{todonotes}
\usepackage{qtree}

\usepackage{todonotes}
\usepackage{qtree}

\newtheorem{theorem}{Theorem}[section]
\newtheorem{proposition}[theorem]{Proposition}
\newtheorem{corollary}[theorem]{Corollary}

\theoremstyle{definition}

\newtheorem{definition}[theorem]{Definition}

\newtheorem{question}[theorem]{Question}
\newtheorem{claim}{Claim}

\numberwithin{equation}{section}

\theoremstyle{definition}
\newtheorem{remark}[theorem]{Remark}
\newtheorem{remarks}[theorem]{Remarks}

\newcommand{\into}{\hookrightarrow}

\renewcommand{\Mc}{\mathcal{M}}
\newcommand{\Oc}{\mathcal{O}}

\newcommand{\Sc}{\mathcal{S}}
\newcommand{\iso}{\cong}

\newcommand{\Ab}{\mathbb{A}}
\newcommand{\A}{\mathbb{A}}
\newcommand{\Cb}{\mathbb{C}}
\newcommand{\C}{\mathbb{C}}
\newcommand{\Eb}{\mathbb{E}}
\newcommand{\Fb}{\mathbb{F}}
\newcommand{\F}{\mathbb{F}}

\newcommand{\Lb}{\mathbb{L}}

\newcommand{\Pb}{\mathbb{P}}
\newcommand{\Qb}{\mathbb{Q}}
\newcommand{\Q}{\mathbb{Q}}

\newcommand{\Zb}{\mathbb{Z}}
\newcommand{\Z}{\mathbb{Z}}
\newcommand{\Fqbar}{\overline{\Fb}_q}
\newcommand{\et}{{et}}

\newcommand{\Aff}{\mathbf{Aff}}

\newcommand{\Ratb}{Rat^\ast}
\newcommand{\Rat}{Rat}

\newcommand{\Poly}{Poly}

\DeclareMathOperator{\PGL}{PGL}

\DeclareMathOperator{\Var}{Var}
\DeclareMathOperator{\PConf}{PConf}
\DeclareMathOperator{\ev}{ev}
\DeclareMathOperator{\Frob}{Frob}

\DeclareMathOperator{\Gal}{Gal}
\DeclareMathOperator{\Fix}{Fix}
\DeclareMathOperator{\tr}{Trace}

\title{Topology and arithmetic of resultants, I}

\author{Benson Farb and Jesse Wolfson}
\thanks{B.F. is supported in part by NSF Grant Nos. DMS-1105643 and DMS-1406209. J.W. is supported in part by NSF Grant No. DMS-1400349.}

\begin{document}
\maketitle

\begin{abstract}
    We consider the interplay of point counts, singular cohomology, \'etale cohomology, eigenvalues of the Frobenius and the Grothendieck ring of varieties for two families of varieties: spaces of rational maps and moduli spaces of marked, degree $d$ rational curves in $\Pb^n$.  We deduce as special cases algebro-geometric and arithmetic refinements of topological computations of Segal, Cohen--Cohen--Mann--Milgram, Vassiliev and others.
\end{abstract}

 \section{Introduction}
The starting point of this paper is the idea that topological theorems about algebraic varieties should have algebro-geometric proofs,  and that such proofs should yield arithmetic information.   More precisely, suppose that $X$ is a (not necessarily projective) algebraic variety defined over $\Z$.  We can then either extend scalars or reduce modulo a prime
$p$ in order to view $X$ through three lenses:

\bigskip
\noindent
{\bf Topological: }The $\C$-points $X(\C)$ form a complex algebraic variety.  Attached to $X(\C)$ are its compactly supported singular cohomology groups $H^*_c(X(\C);\C)$ and Betti numbers $b_i(X(\C))$.

\bigskip
\noindent
{\bf Geometric: }  Let $q=p^d$ be any positive power of $p$.  The $\Fqbar$-points $X(\Fqbar)$, where $\Fqbar$ is the algebraic closure of the finite field $\F_q$.    The set $X(\Fqbar)$ comes equipped with an action of the Frobenius $\Frob_q$, acting on the coordinates of affine charts via $x\mapsto x^q$.
Attached to this setup we have:
\begin{itemize}
\item The associated compactly supported \'{e}tale cohomology groups $H^\ast_{\et,c}(X/\Fqbar;\Qb_\ell)$.  These are representations of the Galois  group $\Gal(\Fqbar/\F_q)$.

\item The eigenvalues of $\Frob_q$ acting on  $H^\ast_{\et,c}(X/\Fqbar;\Qb_\ell)$; these are called {\em weights}.
\end{itemize}

\bigskip
\noindent
{\bf Arithmetic: } The set  $X(\F_q)$ of $\F_q$-points.   As observed by Hasse, this set can be realized as the fixed set of $\Frob_q: X(\Fqbar)\to X(\Fqbar)$.

 \bigskip
 In this paper we consider the interplay of these three viewpoints applied in two concrete situations: spaces of rational maps and moduli spaces of marked, degree $d$ rational curves in $\Pb^n$.
We obtain as special cases algebro-geometric and arithmetic refinements of topological computations of Segal, Cohen--Cohen--Mann--Milgram, Vassiliev and others.  Our main inspiration comes from the ideas in Segal's beautiful paper \cite{Se}.

 \bigskip
 \noindent
 {\bf Spaces of rational maps. }In this paper we consider the following families of spaces of (tuples of) polynomials.

\begin{definition}[{\boldmath $\Poly_n^{d,m}$}]
\label{definition:poly}
Fix a field $K$ with algebraic closure denoted $\overline{K}$.  Fix $d,n\geq 0$ and $m\geq 1$.   Define $\Poly_{n}^{d,m}$ to be the set of all
$m$-tuples $(f_1,\ldots ,f_m)$ of polynomials $f_i\in K[z]$ such that:
\begin{enumerate}
\item Each $f_i$ is monic of degree $d$.
\item The set of polynomials $\{f_1,\ldots ,f_m\}$ has no common root in $\overline{K}$ of multiplicity $n$ or greater.
\end{enumerate}
\end{definition}

The classical theory of discriminants and resultants tells us that $\Poly_n^{d,m}$ is an algebraic variety defined over $\Z$; see \S\ref{section:polyproof} below.  The varieties $\Poly_n^{d,m}$ include well-known classes of varieties as special cases. These include the following examples.
\begin{enumerate}
\item  $\Poly_1^{d,m}$ can be identified with the space $\Rat_{d,m-1}^*$ of all degree $d$ rational maps $\phi: \Pb^1\to \Pb^{m-1}$ that take a given basepoint in $\Pb^1$ to a given
basepoint in $\Pb^{m-1}$.  This is so because  an $m$-tuple of degree $d$ polynomials $(f_0(z),\ldots ,f_{m-1}(z))$  determines, and is determined by, a {\em rational map} $\phi:\Pb^1\to\Pb^{m-1}$ such that $\phi(\infty)=[1:\cdots:1]$ and $\phi^*\Oc(1)\cong \Oc(d)$ via
 \[\phi(z):=[f_0(z):\cdots : f_{m-1}(z)].\]

\item $\Poly_2^{d,1}$ is the space of monic, degree $d$, square-free polynomials, studied in detail by Arnol'd \cite{Arn} and many others.  More generally, $\Poly_n^{d,1}$ is the space of degree $d$ polynomials with no root of multiplicity $n$.  These spaces were studied in detail by Vassiliev \cite{Va} and others.
\end{enumerate}

 In the theorem that follows, $K_0(\Var_K)$ denotes the Grothendieck ring of $K$-varieties and $\Lb:=[\Ab^1] \in K_0(\Var_K)$ denotes the Lefschetz motive; see \S\ref{section:background}.

\begin{theorem}[{\bf Arithmetic of \boldmath$\Poly_n^{d,m}$}]
\label{thm:poly1}
    Let $m\ge 1$ and $d\ge n\ge 1$.
    \begin{enumerate}
        \item \label{thm:points} {\bf Motive/point count: } When ${\rm char}(K)=0$ then
\[[\Poly_n^{d,m}]=\Lb^{dm}-\Lb^{(d-n)m+1}\ \ \text{in}\ \ K_0(\Var_K).\]
                                                Further, for any prime power $q$:
\[|\Poly_n^{d,m}(\Fb_q)|=q^{dm}-q^{(d-n)m+1}.\]

        \item {\bf Betti numbers : }
            \begin{equation*}
                b_i(\Poly_n^{d,m}(\Cb))=\left\{
                    \begin{array}{ll}
                    1 & i=0,~2nm-3\\
                    0 & \text{else}
                    \end{array}\right.
            \end{equation*}
        \item {\bf Comparison Theorem: } There are isomorphisms of graded vector spaces
            \begin{equation*}
                H^\ast_{\et,c}(\Poly_{n/\Fqbar}^{d,m};\Qb_\ell)\otimes_{\Qb_\ell}\Cb\cong H^\ast_c(\Poly_n^{d,m}(\Cb);\Cb),
            \end{equation*}
        \item \label{thm:evalues} {\bf Weights: } There are isomorphisms of Galois representations
            \begin{equation*}
                H^i_{et,c}(\Poly_{n/_{\Fqbar}}^{d,m};\Qb_\ell)\cong\left\{
                    \begin{array}{ll}
                        \Qb_\ell((n-d)m-1) & i=2(d-n)m+3\\
                        \Qb_\ell(-dm) & i=2dm\\
                        0 & \text{else}
                    \end{array}\right.
            \end{equation*}
    \end{enumerate}
\end{theorem}

Our approach to (3) of Theorem \ref{thm:poly1} is not the standard one -- by compactifying with normal crossings divisors.  Instead, we use the combinatorics of a ``root cover'' and the $\ell$-adic cohomology of linear subspace arrangements due to Bj\"{o}rner-Ekedahl \cite{BE}.

The numerics of Theorem \ref{thm:poly1} shows that two rather different-looking varieties have common arithmetic features.

\begin{corollary}[{\bf Comparing \boldmath$\Ratb_{d,n}$ and \boldmath$\Poly^{d(n+1),1}_{n+1}$}]
\label{cor:rat=poly}\mbox{}
    \begin{enumerate}
        \item When ${\rm char}(K)=0$ there are equalities in the Grothendieck ring of varieties
            \begin{align*}
                               [\Ratb_{d,n}]&=[\Poly^{d(n+1),1}_{n+1}].
            \end{align*}
            Further, $\Ratb_{d,n}$ and $\Poly^{d(n+1),1}_{n+1}$ have equal point counts over all finite fields.
        \item There are isomorphisms of graded Galois representations
            \begin{align*}
                               H^\ast_{et,c}(\Ratb_{d,n/_{\Fqbar}};\Qb_\ell)&\cong H^\ast_{et,c}(\Poly^{d(n+1),1}_{n+1/_{\Fqbar}};\Qb_\ell).
            \end{align*}
    \end{enumerate}
\end{corollary}

These results have topological precursors. Cohen--Cohen--Mann--Milgram \cite{CCMM} showed that for any generalized homology theory $\Eb_\ast$,
\begin{equation*}
    \Eb_\ast(\Ratb_{d,1}(\Cb))\cong\Eb_\ast(\Poly^{2d,1}_2(\Cb)).
\end{equation*}
Building on this, Vassiliev \cite{Va} showed that
\begin{equation*}
    \Eb_\ast(\Ratb_{d,n}(\Cb))\cong\Eb_\ast(\Poly^{d(n+1),1}_{n+1}(\Cb)).
\end{equation*}

Guest--Kozlowski--Yamaguchi \cite{GKY} showed that, for $n>1$, Vassiliev's isomorphism is actually induced by a homotopy equivalence
\begin{equation*}
    \Ratb_{d,n}(\Cb)\simeq \Poly^{d(n+1),1}_{n+1}(\Cb).
\end{equation*}

Corollary \ref{cor:rat=poly} lifts these equivalences to the $\ell$-adic cohomology of the associated varieties over finite fields.  Are there algebraic explanations for all of this?

\begin{question}
\label{question:iso}
Are the varieties $\Ratb_{d,n}$ and $\Poly_{n+1}^{d(n+1),1}$ isomorphic for $n\geq 2$? If not, what invariant distinguishes them? \footnote{H. Spink and D. Tseng have recently answered this question in the negative; see \cite{ST}.}
\end{question}

We would further like to know the following.

\begin{question}
Does the equality $[\Poly_n^{d,m}]=\Lb^{dm}-\Lb^{(d-n)m+1}\ \ \text{in}\ \ K_0(\Var_K)$ of Theorem~\ref{thm:poly1} and $[\Ratb_{d,n}]=[\Poly^{d(n+1),1}_{n+1}]$ of Corollary~\ref{cor:rat=poly} hold when ${\rm char}(K)>0$?  We will see below that in each case there is a bijective morphism between the corresponding varieties; when ${\rm char}(K)=0$ this implies the desired equalities in $K_0(\Var_K)$.
\end{question}

\begin{remarks}\mbox{}
\begin{enumerate}
\item Note that the requirement that $n\geq 2$ in Question~\ref{question:iso} is necessary: it is {\em not} the case that $\Ratb_{d,1}$ is homotopy equivalent to $\Poly_{2}^{2d,1}$, since $\pi_1(\Ratb_{d,1}(\C))\iso \Z$ but $\Poly_{2}^{2d,1}(\C)\iso B_{2d}$, the braid group on $2d$ strands.

\item After seeing an earlier draft of this paper, Curt McMullen gave the following argument to show that Question~\ref{question:iso} has a positive answer in the first nontrivial case, $d=1$ and $n=2$, as follows.  The space $\Poly^{3,1}_3$ is the complement in $\A^3$ of a twisted cubic, since $(x+a)^3 = x^3 + 3ax^2 + 3a^2x + a^3$.   On the other hand, the space $\Ratb_{1,2}$ is isomorphic to $\A^3- \A^1$.   Over $\Z[\frac{1}{3}]$, it is an exercise to write down an explicit isomorphism (involving division by $3$) between these varieties.  In characteristic $3$, the twisted cubic above becomes $\{(0,0,a^3)\}$, and its complement is isomorphic to $\A^3- \A^1$ via the Frobenius $x\mapsto x^3$.
\end{enumerate}
\end{remarks}

\bigskip
\noindent
{\bf{The moduli space of $m$-pointed rational curves in $\Pb^n$}.}
Fix $d,n\geq 1$ and $m\ge 0$.  Let $\Mc_{0,m}(\Pb^n,d)$ be the {\em moduli space of $m$-pointed, degree $d$ rational curves in $\Pb^n$};  see \S\ref{section:M0m} below for a precise definition.  The spaces $\Mc_{0,m}(\Pb^n,d)$ are schemes defined over $\Z$.  We state our results for the associated subvariety $\Mc_{0,m}^\ast (\Pb^n,d)$ of curves passing through a fixed base point in $\Pb^n$.

\begin{theorem}[{\bf Arithmetic of \boldmath$\Mc^\ast_{0,m}(\Pb^n,d)$}]
\label{theorem:M0m}
    Let $m\ge 3$ and let $d,n\ge 1$. Then

    \begin{enumerate}
        \item {\bf Motive/point count: }
         When ${\rm char}(K)=0$ then
            \begin{equation*}
                [\Mc^\ast_{0,m}(\Pb^n,d)]=\left(\prod_{i=2}^{m-2}(\Lb-i)\right)\left(\Lb^{d(n+1)}-\Lb^{(d-n-1)(n+1)+1}\right).
            \end{equation*}
            in $K_0(\Var_K)$. Further, for any prime power $q$:
            \[
                |\Mc^\ast_{0,m}(\Pb^n,d)(\Fb_q)|=\left(\prod_{i=2}^{m-2}(q-i)\right)\left(q^{d(n+1)}-q^{(d-n-1)(n+1)+1}\right)
            \]
        \item {\bf Comparison Theorem: } We have isomorphisms of graded vector spaces
            \begin{equation*}
                H^\ast_{\et,c}(\Mc_{0,m}^\ast(\Pb^n;d)_{/\Fqbar};\Qb_\ell)\otimes_{\Qb_\ell}\Cb\cong H^\ast_c(\Mc_{0,m}^\ast(\Pb^n;d)(\Cb);\Cb).
            \end{equation*}
        \item {\bf Weights: } Let
            \begin{equation*}
                \nu(a):=\sum_\sigma\prod_{j=1}^{2(m-3)-a}(\sigma(j)+2),
            \end{equation*}
            where the sum is over order preserving injections
            \begin{equation*}
                \sigma\colon\{1,\ldots,2(m-3)-a\}\into\{0,\ldots,m-4\}.
            \end{equation*}
            Then, as a Galois representation, $H^i_{\et,c}(\Mc^\ast_{0,m}(\Pb^n,d)_{/\Fqbar};\Qb_\ell)$ consists precisely of
            \begin{enumerate}
                \item a direct summand
                    \begin{equation*}
                        \Qb_\ell(m-3+d(n+1)-i)^{\oplus\nu(2(m-3+d(n+1))-i))}
                    \end{equation*}
                    if $m-3+2d(n+1)\le i\le 2(m-3+d(n+1))$, and
                \item a direct summand
                    \begin{equation*}
                        \Qb_\ell(m-1+d(n+1)-i)^{\oplus \nu(2(m-2+d(n+1))-(i+n))}
                    \end{equation*}
                    if $m+(2d-1)(n+1)\le i\le 2(m-3+d(n+1))-n+2$,
            \end{enumerate}
            and nothing else.
    \end{enumerate}
\end{theorem}

\paragraph{Acknowledgements} We thank Sasha Beilinson, Weiyan Chen, Jordan Ellenberg, Matt Emerton, Ezra Getzler, Sean Howe, Peter May, June Park, Joel Specter, Shmuel Weinberger, and Melanie Wood for helpful conversations. We thank Nir Gadish for explaining to us the argument in the first part of the proof of Theorem \ref{theorem:M0m}.   We would like to acknowledge that a post on Jordan Ellenberg's blog, Quomodocumque, based on an idea of Michael Zieve, helped to inspire one of the arguments in this paper.  We thank Ben O'Connor for corrections on an earlier draft of this paper. We are grateful to Curt McMullen and Joe Silverman for numerous comments, corrections and suggestions on an earlier draft of this paper. Finally, we are grateful to Hunter Spink and Dennis Tseng for pointing out a mistake in an earlier version of this paper (cf.\ Remark~\ref{remark:ST} below).

\section{Preliminaries}
\label{section:background}
In this section we review some background material that we will use later in the paper.

\paragraph{The Grothendieck ring of varieties.}
The Grothendieck ring $K_0(\Var_K)$ of varieties over the field $K$ gives a useful framework for computing invariants like point counts. Recall that, as a group, $K_0(\Var_K)$ is generated by isomorphism classes of $K$-varieties $[X]$, modulo the relations $[X]=[Z]+[X-Z]$ for $Z\subset X$ a closed subvariety. Multiplication is induced by the Cartesian product of varieties, i.e. $[X][Y]=[X\times_K Y]$. It is elementary to prove that, for $K=\F_q$, the assignment
\begin{equation*}
    [X]\mapsto |X(\F_q)|
\end{equation*}
extends to a homomorphism $K_0(\Var_K)\to \Zb$. Thus, computations in $K_0(\Var_K)$ carry implications for point counts. Following convention, we denote by $\Lb$ the class $[\Ab^1]\in K_0(\Var_k)$ and refer to it as the \emph{Lefschetz motive}.

\paragraph{The Grothendieck--Lefschetz Trace Formula}
Let $Z$ be a scheme defined over $\Z$.  We can reduce modulo $q$ for any prime power $q$ to obtain an algebraic variety defined over $\F_q$.  We can then ask for the number $|Z(\F_q)|$ of $\F_q$-points of $Z$.

How can we compute $|Z(\F_q)|$? One approach begins with Hasse's fundamental observation that the set $Z(\F_q)$ is the set of fixed points of the {\em geometric Frobenius} morphism $\Frob_q\colon Z\to Z$. In topology, the classical Lefschetz Fixed Point Theorem computes (in many cases) the number of fixed points of a continuous self-map $f:Z\to Z$ of a triangulable topological space in terms of the traces of the induced maps $f^*:H^i(Z;\Q)\to H^i(Z;\Q)$ on the singular cohomology of $Z$.  While singular cohomology is unsuitable for studying varieties over $\Fqbar$, the {\em \'{e}tale cohomology}, developed by Grothendieck and his school, is designed precisely for this purpose.

To set this up, fix a prime power $q$, a prime $\ell$ not dividing $q$, and let $\Q_\ell$ denote the $\ell$-adic rationals.  Attached to any variety $Z$ defined over $\F_q$ are its \'{e}tale cohomology groups $H^i_{\et}(X_{/\Fqbar};\Q_\ell), i\geq 0$ (cf.\ e.g.\ \cite{Mi} or \cite{De}). The key result we will use from this theory is the {\em Grothendieck--Lefschetz Trace Formula} \cite[Theorem 25.1]{Mi}.  For smooth projective varieties $Z$ defined over $\F_q$, this formula gives:
\begin{equation}
\label{eq:GL1}
\begin{array}{ll}
\left |Z(\F_q)\right|&=\#\Fix(\Frob_q:Z(\Fqbar)\to Z(\Fqbar)) \\
&\\
&= \sum_{i\geq 0}(-1)^i\tr\big(\Frob_q:H^i_{\et}(Z_{/\Fqbar};\Q_\ell)\to H^i_{\et}(Z_{/\Fqbar};\Q_\ell)\big)
\end{array}
\end{equation}
However the varieties we consider in this paper are not projective. To remedy this, we first note that Formula \eqref{eq:GL1} holds for any $Z$ of finite type if we replace $H^i_{\et}(Z_{/\Fqbar};\Q_\ell)$ by compactly supported \'{e}tale cohomology $H^i_{\et,c}(Z_{/\Fqbar};\Q_\ell)$ (cf.  \cite[6.1.1.1]{De}).  When $Z$ is smooth, we can then apply Poincar\'{e} duality for \'{e}tale cohomology \cite[Theorem 24.1]{Mi} to obtain
\[H^i_{\et,c}(Z_{/\Fqbar};\Q_\ell)\iso H^{2\dim(Z)-i}_{\et}(Z_{/\Fqbar};\Q_\ell(-\dim(Z)))^*\]
where $^*$ denotes the dual space.  Plugging this in to \eqref{eq:GL1} gives, for any smooth but not necessarily projective variety:
\begin{equation}
\label{eq:GL2}
\left|Z(\F_q)\right|= q^{\dim(Z)}\sum_{i\geq 0}(-1)^i
\tr\big(\Frob_q:H^i_{\et}(Z_{/\Fqbar};\Q_\ell)^{*}\to H^i_{\et}(Z_{/\Fqbar};\Q_\ell)^{*}\big)
\end{equation}

\section{Proof of Theorem~\ref{thm:poly1}}
\label{section:polyproof}

We remark that when $n>d$, the condition of having no common root of multiplicity $n$ is clearly empty.  When $n=d$, the condition that the degree $d$ polynomials $f_i$ have a common root of multiplicity $n=d$ is simply that there exists $z_0\in K$ so that $f_i(z)=(z-z_0)^d$ for each $i$; the space of such polynomials is thus isomorphic to $\A^1$.  We thus have
\begin{equation}
    \label{eq:spfiltration0}
        \Poly_n^{d,m}\cong\left\{\
            \begin{array}{ll}
                \Ab^{dm}&\text{if $n>d$}\\
                \Ab^{dm}-\Ab^1&\text{if $n=d$}
            \end{array}\right.
\end{equation}

Thus the most interesting case is when $d>n$.

\begin{proof}[Proof of Theorem~\ref{thm:poly1}]
We prove the theorem by filtering the space of $m$-tuples of monic, degree $d$ polynomials
by closed subvarieties, and analyzing this filtration via topology and algebraic geometry, in a series of steps.

\bigskip
\noindent
{\bf Step 1 (Building a filtration): }  Recording coefficients gives an isomorphism from the space of $m$-tuples $(f_1,\ldots ,f_m)$ of monic, degree $d$ polynomials to the affine space $\Ab^{dm}$.  Parameterizing
this space by the ``roots'' of the polynomials in the ordered $m$-tuple, we see that the ordered space of ``roots'' is isomorphic to $\overbrace{\A^d\times\cdots \times \A^d}^m$,
and that $S_d\times \cdots \times S_d  \ \text{($m$ times)}$ acts on this variety by permuting the roots.  The quotient $\A^d\times\cdots \times \A^d/(S_d\times\cdots\times S_d)$ is thus a variety, and in fact is isomorphic to $\A^{dm}$, by Newton's theorem on symmetric polynomials.\footnote{We remind the reader that the identification of the space of $m$-tuples of monic polynomials with the quotient $\Ab^{dm}/(S_d)^m$ is an isomorphism of schemes. Over an algebraically closed field $\bar{K}$, the quotient map $\Ab^{dm}\to\Ab^{dm}/(S_d)^m$ will be surjective on $\bar{K}$-points (since the roots of any polynomial over $\bar{K}$ are also in $\bar{K}$), but this will not be the case over a general field.}

For any $k\geq 0$, denote by $R_{n,k}^{d,m}$ the space of $m$-tuples $(f_1,\ldots ,f_m)$ of monic, degree $d$ polynomials for which there exists a monic $h\in K[z]$ with $\deg(h)\geq k$ and monic polynomials $g_i\in K[z]$ so that
    \[f_i(z)=g_i(z)h(z)^n\]
    for each $1\leq i\leq m$.  So for example $R_{n,0}^{d,m}=\A^{dm}$, and the $R^{d,m}_{n,k}$ give a descending filtration
    \begin{equation}
    \label{eq:filtration}
        \A^{dm}=R_{n,0}^{d,m}\supset R_{n,1}^{d,m}\supset\cdots \supset \emptyset.
    \end{equation}

We claim that each $R_{n,k}^{d,m}$ is a closed subvariety of $\A^{dm}$.
To see this, let $\Sc_{d,m,n,k}$ be the collection of all
$m$-tuples
$\sigma=(\sigma_1,\ldots ,\sigma_m)$ of injections
\[\sigma_i:\{1,\ldots ,n\}\times \{1,\ldots ,k\} \to \{1,\ldots d\}\]
with the property that each $\sigma_i$ is order-preserving on each $\{1,\ldots ,n\}\times \{j\}$
and such that $\sigma_i(1,a)<\sigma_i(1,b)$ for each $a<b$.  For each $\sigma\in \Sc_{d,m,n,k}$,
let $L_\sigma$ be the linear subspace of $\A^{dm}$ defined by the equations

\[\{x_{\sigma_i(a,b)}=x_{\sigma_i(a',b)}: 1\leq a,a'\leq n,~ 1\le i\le m\} \cup \{x_{\sigma_i(a,b)}=x_{\sigma_j(a,b)}: 1\leq i,j\leq m\}.\]
Note that each $L_\sigma$ is an affine subspace of $\A^{dm}$ of dimension $k+m(d-nk)$.  The action of $S_d^m$ on $\A^d\times\cdots\times\A^d$ preserves the union of linear subspaces $\bigcup_{\sigma\in \Sc_{d,m,n,k}} L_\sigma$, and $R_{n,k}^{d,m}$ is the quotient of the union of linear subspaces by this action.   Since the quotient of an affine variety by a finite group action is an affine variety, and since such quotient maps take closed invariant subvarieties to closed subvarieties, it follows that each $R_{n,k}^{d,m}$ is an closed subvariety of $\A^{dm}$.

\medskip
\noindent
{\bf Step 2 (Extracting common factors): } Let $k\geq 0$.  Define a morphism
\[\overline{\Psi}:\A^{m(d-nk)}\times\A^k\to \A^{md}\]
by
\[\overline{\Psi}(f_1,\ldots,f_m,g):=(f_1g^n,\ldots,f_mg^n).\]
The restriction of $\overline{\Psi}$ to $\Poly_n^{d-kn,m}\times \A^k$ gives a morphism
\begin{equation}
\label{eq:spfiltration2}
\Psi: \Poly_n^{d-kn,m}\times \A^k\to  R_{n,k}^{d,m}-R_{n,k+1}^{d,m}
    \end{equation}
where the target is the space of $m$-tuples of degree $d$ polynomials with a common $n$-fold factor of degree equal to $k$, with no other common $n$-fold factors.  We think of the map $\Psi^{-1}$ as the (non-algebraic) map that extracts a common $n$-fold factor from a tuple of polynomials.   We claim that:

\begin{enumerate}
\item For any field $k$ the morphism $\Psi$ is bijective.
\item For $k=\Cb$, the map $\Psi$ is a homeomorphism in the classical topology.
\end{enumerate}

These facts will allow us to analyze $\Poly^{d,m}_n$ recursively.  Note that the case $k=0$ follows by definition:
    \begin{equation*}
       \Poly_n^{d,m}:=R_{n,0}^{d,m}-R_{n,1}^{d,m}.
    \end{equation*}

To see (1): It is clear from the definitions that $\Psi$ is surjective.  The map $\Psi$ is injective because there is a unique $n$-fold degree $k$ factor in each $f_ig^n$, so if $f_ig^n=u_iv^n$ then this implies $g=v$ and so $f_i=u_i$.

To see (2): First note that the spaces of polynomials
in the range and domain of $\overline{\Psi}$ have Galois covers given by the corresponding spaces
of (all possible orderings of) roots, with deck group the appropriate symmetric group.  The map $\overline{\Psi}$
lifts to a map between these spaces of roots :

\[\Phi:A^{m(d-nk)}\times\A^k\to \A^{md}\]
\noindent
given by

\[\Phi((\vec{r}_1,\ldots,\vec{r}_m),\vec{s})):=((\vec{r}_1,(\vec{s})^n),\ldots , (\vec{r}_m,(\vec{s})^n))\]
\noindent
where $\vec{r}_i$ is the vector of $d$ roots of $f_i$;  the vector of roots of $g$ is denoted $\vec{s}$; and where
$(\vec{s})^n$ denotes the vector $(\vec{s},\ldots ,\vec{s})$, where $\vec{s}$ is repeated $n$ times.  It follows that the map $\Phi$ is closed, and hence the map $\overline{\Psi}$ is closed, and hence the map $\Psi$ is closed.  Since $\Psi$ is bijective, it follows that $\Psi$ is a homeomorphism.

\begin{remark}
\label{remark:ST}
One might hope that the bijective morphism $\Psi$ is in fact an isomorphism, and indeed this was claimed in an earlier version of this paper.  However, as pointed out to us by H. Spink and D. Tseng, this is not true.
\end{remark}

\medskip
\noindent
{\bf Step 3 (Computing \boldmath$[\Poly_n^{d,m}]\in K_0(\Var_K)$ and $|\Poly_n^{d,m}(\Fb_q)|$): } We have shown that there is a descending filtration of closed subvarieties:
    \begin{equation*}
        \A^{dm}=R_{n,0}^{d,m}\supset R_{n,1}^{d,m}\supset\cdots \supset \emptyset.
    \end{equation*}
    As a result, $\A^{dm}$ admits a disjoint decomposition by locally closed subvarieties
    \begin{equation*}
        \A^{dm}=\coprod_{k\geq 0} (R_{n,k}^{d,m}- R_{n,k+1}^{d,m}).
    \end{equation*}
    Taking classes in the Grothendieck ring $K_0(\Var_K)$ gives
    \begin{equation}\label{eq:spfiltration1}
        \Lb^{dm}=\sum_{k\geq 0} ([R_{n,k}^{d,m}]-[R_{n,k+1}^{d,m}]),
    \end{equation}
    where we write $\Lb$ for the class $[\Ab^1]\in K_0(\Var_K)$.

We now claim that, when ${\rm char}(K)=0$ then
\begin{equation}
\label{eq:spfiltration2:GR}
[\Poly_n^{d-kn,m}]\cdot \Lb^k= [R_{n,k}^{d,m}] - [R_{n,k+1}^{d,m}]
\end{equation}
To see this, first note that we proved in Step 2 that the map $\Psi$ in \eqref{eq:spfiltration2} is a bijective morphism on $K$-points for all fields $K$.  It is known (see, e.g., Remark 4.1 of \cite{Go}) that if ${\rm char}(K)=0$ then a bijective morphism of $K$-varieties induces an equality $[X]=[Y]$ in the Grothendieck ring of $K$-varieties.

Plugging in the expression from Equation \eqref{eq:spfiltration2:GR}  into Equation \eqref{eq:spfiltration1} then gives the following  recursive formula in the ring $K_0(\Var_K)$ when ${\rm char}(K)=0$:
    \begin{equation}\label{eq:spfiltration3}
        [\Poly_n^{d,m}]=\Lb^{dm} - \sum_{k\geq 1}[\Poly_n^{d-kn,m}]\cdot \Lb^k
    \end{equation}
    It is left to prove the claimed result, namely that for $d>n\geq 1$, this recursion is solved by
    $[\Poly_n^{d,m}]=\Lb^{dm}-\Lb^{(d-n)m+1}$.  We proceed by induction on $d$.   This gives :
    \begin{equation*}
        \begin{array}{ll}
            [\Poly_n^{d,m}]&=\Lb^{dm} - \sum_{k\geq 1}[\Poly_n^{d-kn,m}]\cdot \Lb^k\\
            &\\
            &=\Lb^{dm}-(\sum_{k\geq 1}^{\lfloor \frac{d}{n}\rfloor-1}(\Lb^{(d-nk)m}-\Lb^{(d-n(k+1))m+1})\cdot \Lb^k+ \Lb^{(d-n\lfloor\frac{d}{n}\rfloor)m}\cdot\Lb^{\lfloor \frac{d}{n}\rfloor})\\
            &\\
            &=\Lb^{dm}-(\sum_{k\geq 1}^{\lfloor \frac{d}{n}-1\rfloor}(\Lb^{(d-nk)m+k}-\Lb^{(d-n(k+1))m+k+1})+ \Lb^{(d-n\lfloor\frac{d}{n}\rfloor)m+\lfloor \frac{d}{n}\rfloor})\\
            &\\
            &=\Lb^{dm}-(\Lb^{(d-n)m+1}-\Lb^{(d-n\lfloor\frac{d}{n}\rfloor)m+\lfloor \frac{d}{n}\rfloor}+\Lb^{(d-n\lfloor\frac{d}{n}\rfloor)m+\lfloor \frac{d}{n}\rfloor})\\
            &\\
            &=\Lb^{dm}-\Lb^{(d-n)m+1}.
        \end{array}
    \end{equation*}

When $K=\Fb_q$, by replacing each instance of $[X]\in K_0(\Var_K)$ by $|X(\F_q)|$ in the argument above, we conclude that $\Poly_n^{d,m}(\F_q)=q^{dm}-q^{(d-n)m+1}$.  This proves Statement (1) of the theorem.

\medskip
\noindent
{\bf Step 4 (The comparison theorem): }
 We now establish Statement (3). Artin's comparison theorem (see \cite{Art}) shows that the $\ell$-adic cohomology of the variety $\Poly_{n/\Cb}^{d,m}$ agrees with the singular cohomology of $\Poly_n^{d,m}(\C)$.  We now claim that this agreement persists when we replace $\Cb$ by $\Fqbar$. To see this, we will use transfer plus a result of Bj\"orner--Ekedahl on complements of linear subspace arrangements.

    The relevant subspace arrangement $V_{\mathcal{A}}\subset\Ab^{dm}$ consists of all subspaces of the form
    \begin{equation*}
        L_{\bar{\sigma}}=\{z_{1,\sigma_1(i)}=\cdots=z_{m,\sigma_m(i)}~|~1\le i\le n\}
    \end{equation*}
    for some collection of injections
    \begin{equation*}
        \bar{\sigma}=\{\{1,\ldots,n\}\to^{\sigma_j}\{1,\ldots,d\}\}_{j=1}^m.
    \end{equation*}
    The action of $(S_d)^{\times m}$ on $\Ab^{dm}$ by permuting the coordinates preserves the arrangement $V_{\mathcal{A}}$, and, over an algebraically closed field $\overline{K}$, the variety $\Poly_n^{d,m}$ is the quotient of this action on the complement $\Ab^{dm}-V_{\mathcal{A}}$. By transfer, we see that
    \begin{equation*}
        H^i_{et,c}(\Poly_{n/\bar{K}}^{d,m};\Qb_\ell)\cong H^i_{et,c}(\Ab^{dm}_{\bar{K}}-V_{\mathcal{A}/\bar{K}};\Qb_\ell)^{(S_d)^{\times m}}.
    \end{equation*}
    Further, if we denote by $L_{\mathcal{A}}$ the intersection lattice of the subspaces in $V_{\mathcal{A}}$, we see that the natural identification
    \begin{equation*}
        L_{\mathcal{A}_{/\Cb}}\to L_{\mathcal{A}_{/\Fqbar}}
    \end{equation*}
    defines an $(S_d)^{\times m}$-equivariant isomorphism of lattices which also respects the natural dimension functions on each. By \cite[Theorem 4.9(i)]{BE}, the $\ell$-adic cohomology of the complement of a subspace arrangement is functorially determined by the intersection lattice together with its dimension function. In particular, the isomorphism above defines an $(S_d)^{\times m}$-equivariant isomorphism
    \begin{equation*}
        H^i_{et,c}(\Ab^{dm}_{\Cb}-V_{\mathcal{A}/\Cb};\Qb_\ell)\cong H^i_{et,c}(\Ab^{dm}_{\Fqbar}-V_{\mathcal{A}/\Fqbar};\Qb_\ell).
    \end{equation*}
The restriction of this isomorphism to the subspaces of invariants gives the claimed
comparison isomorphism, proving Statement (3).

\medskip
\noindent
{\bf Step 5 (Betti numbers): }To prove Statement (2) we proceed by induction on $d$. For the base case, the statement of the theorem follows immediately from the isomorphism
    \begin{equation*}
        \Poly_n^{n,m}\cong\Ab^{nm}-\Ab^1.
    \end{equation*}
    Now suppose that we have shown the result for $j<d$. We will prove the induction step by first computing the Betti numbers at step $d$.  We then use the comparison isomorphism to obtain the ranks of the $\ell$-adic cohomology groups, and we deduce the weights from the point count and Grothendieck--Lefschetz.

\smallskip
\noindent
{\bf Step 5a (Inducting on the degree): }
    Our argument is an extension of the arguments in Segal \cite{Se}.  As in \textit{loc. cit.} we construct, for all $d>n$, a continuous open embedding
    \begin{align}\label{stabS}
        \Poly_n^{d-1,m}(\Cb)\times\Cb^m\to\Poly_n^{d,m}(\Cb)
    \end{align}
    by ``bringing zeroes in from infinity''. We will show by induction that this induces an isomorphism on compactly supported rational cohomology. For the base case, we have a map of cofiber sequences
    \begin{equation*}
        \xymatrix{
            R_{n,1}^{n+1,m}(\Cb)^+ \ar[r] \ar[d] & R_{n,0}^{n+1,m}(\Cb)^+ \ar[r] \ar[d] & (\Poly_n^{n+1,m}(\Cb))^+ \ar[d]\\
            (R_{n,1}^{n,m}(\Cb)\times\Cb^m)^+ \ar[r] & (R_{n,0}^{n,m}(\Cb)\times\Cb^m)^+ \ar[r] & (\Poly_n^{n,m}(\Cb)\times\Cb^m)^+
        }
    \end{equation*}
    where $X^+$ denotes the 1-point compactification of $X$. This is isomorphic to
    \begin{equation*}
        \xymatrix{
            (\Cb^{m+1})^+ \ar[r] \ar[d] & (\Cb^{(n+1)m})^+ \ar[r] \ar[d] & (\Poly_n^{n+1,m}(\Cb))^+ \ar[d]\\
            (\Cb^1\times\Cb^m)^+ \ar[r] & (\Cb^{nm}\times\Cb^m)^+ \ar[r] & (\Poly_n^{n,m}(\Cb)\times\Cb^m)^+
        }
    \end{equation*}
    Because the first two vertical maps induce isomorphisms in compactly supported cohomology, the Five Lemma (applied to the map of long exact sequences in cohomology) shows that the right vertical map induces a cohomology isomorphism as well.

    \smallskip
\noindent
{\bf Step 5b (Computing  \boldmath$H^i_c(R_{n,k}^{d,m}(\Cb);\Cb)$): }Suppose that we have shown that \eqref{stabS} induces an isomorphism in compactly supported singular rational cohomology for $j<d$. We will deduce the singular cohomology of $\Poly_n^{d,m}(\Cb)$ from the following claim.
    \begin{claim}
        Let $d\ge n$ and $k\le \lfloor \frac{d}{n}\rfloor$. Then the compactly supported singular cohomology $R_{n,k}^{d,m}(\C)$ is given by
        \begin{equation*}
            H^i_c(R_{n,k}^{d,m}(\Cb);\Cb)\cong
                \left\{
                    \begin{array}{ll}
                        \Cb & i=2(d-kn)m+2k\\
                        0 & \text{else}
                    \end{array}\right.
        \end{equation*}
    \end{claim}
    We prove this claim by downward induction on $k$. For the base case, observe that
    \begin{align*}
        R_{n,1}^{n,m}&\cong\Ab^1\\
        R_{n,0}^{n,m}&\cong\Ab^{nm}
    \end{align*}
    so the statement follows. Similarly to \eqref{stabS}, we also construct a continuous open embedding
    \begin{align}\label{stabF}
        R_{n,k}^{d-1,m}(\Cb)\times\Cb^m\to R_{n,k}^{d,m}(\Cb)
    \end{align}
    by ``bringing in zeroes from infinity''. For $d\le n+1$ and all $k$, we see that this induces an isomorphism on compactly supported cohomology.  Similarly, supposing that $n\nmid d$, we see that the map induces an isomorphism on compactly supported cohomology for $k=\lfloor\frac{d}{n}\rfloor$.

    Continuing to assume to $n\nmid d$, we now induct down on $k$. Suppose we have shown the claim for all $j<d$, and also assume that we have shown that the maps \eqref{stabF} induce isomorphisms in compactly supported rational cohomology for $j<d$ and all $k$. For the base case of the induction on $k$, let $a=\lfloor\frac{d}{n}\rfloor$. Then we have
    \begin{align*}
        R_{n,a}^{d,m}&\cong\Poly_n^{d-an,m}\times\Ab^a\\
        &\cong \Ab^{(d-an)m+a},
    \end{align*}
    so the statement follows. Similarly, we observed above that the map \eqref{stabF} induces an isomorphism on compactly supported cohomology for $j=d$ and $k=a$. Now suppose we have shown that \eqref{stabF} induces such an isomorphism for $j=d$ and $k+1>1$. Observe that the ``bringing in zeroes'' maps fit together to give a continuous map of cofiber sequences
    \begin{equation*}
        \xymatrix{
            R_{n,k+1}^{d,m}(\Cb)^+ \ar[r] \ar[d] & R_{n,k}^{d,m}(\Cb)^+ \ar[r] \ar[d] & (\Poly_n^{d-kn,m}(\Cb)\times\Cb^k)^+ \ar[d]\\
            (R_{n,k+1}^{d-1,m}(\Cb)\times\Cb^m)^+ \ar[r] & (R_{n,k}^{d-1,m}(\Cb)\times\Cb^m)^+ \ar[r] & (\Poly_n^{d-1-kn,m}(\Cb)\times\Cb^k\times\Cb^m)^+
        }
    \end{equation*}
    This gives rise to a map of long exact sequences
    \begin{equation*}
        \xymatrix{
            \cdots H^{i-2(k+m)}_c(\Poly_n^{d-1-kn,m}(\Cb);\Cb) \ar[r]\ar[d] & H^{i-2m}_c(R_{n,k}^{d-1,m}(\Cb);\Cb) \ar[r]\ar[d] & H^{i-2m}_c(R_{n,k+1}^{d-1,m}(\Cb);\Cb) \ar[d] \cdots\\
            \cdots H^{i-2k}_c(\Poly_n^{d-kn,m}(\Cb);\Cb) \ar[r] & H^i_c(R_{n,k}^{d,m}(\Cb);\Cb) \ar[r] & H^i_c(R_{n,k+1}^{d,m}(\Cb);\Cb) \cdots
        }
    \end{equation*}
    Our inductive hypotheses and the Five Lemma show that the claim holds for $R_{n,k}^{d,m}$ and that the map \eqref{stabF} is an equivalence for $k$. This concludes the induction step, and thus the claim, when $n\nmid d$.

    When $d=an$ for $a>1$, the induction proceeds as above, once we establish the cases $k=a$ and $k=a-1$. The claim about the cohomology of $R_{n,a}^{an,m}$ follows from the isomorphism
    \begin{equation*}
        R_{n,a}^{an,m}\cong \Ab^a.
    \end{equation*}
    For $k=a-1$, the identification
    \begin{equation*}
        R_{n,a-1}^{an,m}-R_{n,a}^{an,m}\cong \Poly_n^{n,m}\times\Ab^{a-1}\cong (\Ab^{nm}-\Ab^1)\times\Ab^{a-1}
    \end{equation*}
    gives rise to the long exact sequence in compactly supported cohomology
    \begin{equation*}
        \cdots\to H^{i-2(a-1)}_c(\Cb^{nm}-\Cb^1;\Cb)\to H^i_c(R_{n,a-1}^{an,m}(\Cb);\Cb)\to H^i_c(\Cb^a;\Cb)\to^\partial\cdots
    \end{equation*}
    This implies that
    \begin{equation*}
        H^i_c(R_{n,a-1}^{an,m}(\Cb);\Cb)\cong
            \left\{
                \begin{array}{ll}
                    0 & i<2a\\
                    0 & 2a+1<i<2nm+2(a-1)\\
                    \Cb & i=2nm+2(a-1)\\
                    0 & i>2nm+2(a-1)
                \end{array}\right.
    \end{equation*}
    For the remaining cases, we have a long exact sequence
    \begin{align*}
        0\to H^{2a}_c(R_{n,a-1}^{an,m}(\Cb);\Cb)\to H^{2a}_c(\Cb^a;\Cb)\to^\partial &H^{2a+1}_c((\Cb^{nm}-\Cb^1)\times\Cb^{a-1};\Cb)\\
            &\to H^{2a+1}_c(R_{n,a-1}^{an,m}(\Cb);\Cb)\to 0.
    \end{align*}
    It suffices to show that the boundary map is an isomorphism. To see this, consider the closed embedding
    \begin{align*}
        \Ab^1\times\Ab^{a-1}&\to \Ab^{nm}\times\Ab^{a-1}\\
        (z-\lambda,h)&\mapsto((z-\lambda)^n,\cdots,(z-\lambda)^n,h)
    \end{align*}
    where we view $\Ab^1\times\Ab^{a-1}$ as the variety of pairs of monic polynomials $(f,h)$ with $\deg(f)=1$ and $\deg(h)=a-1$, and where we view $\Ab^{nm}\times\Ab^{a-1}$ as the variety of $m+1$-tuples of monic polynomials
    \begin{equation*}
        (f_1,\cdots,f_m,h)
    \end{equation*}
    with $\deg(f_i)=n$ and $\deg(h)=a-1$. By inspection,
    \begin{equation*}
        \Ab^{nm}\times\Ab^{a-1}-\Ab^1\times\Ab^{a-1}\cong\Poly_n^{n,m}\times\Ab^{a-1}
    \end{equation*}
    and the assignments
    \begin{align*}
        (z-\lambda,h)&\mapsto(z-\lambda)h\\
        (f_1,\cdots,f_m,h)&\mapsto(f_1h^n,\cdots,f_mh^n)
    \end{align*}
    determine a map of cofiber sequences
    \begin{equation*}
        \xymatrix{
            (\Cb^1\times\Cb^{a-1})^+ \ar[r] \ar[d] & (\Cb^{nm}\times\Cb^{a-1})^+ \ar[r] \ar[d] & (\Poly_n^{n,m}(\Cb)\times\Cb^{a-1})^+ \ar[d]^\cong \\
            R_{n,a}^{an,m}(\Cb)^+ \ar[r] & R_{n,a-1}^{an,m}(\Cb)^+ \ar[r] & (\Poly_n^{n,m}(\Cb)\times\Cb^{a-1})^+
        }
    \end{equation*}
    The left vertical map is an $a$-fold branched cover, so on the top degree of compactly supported cohomology, the map it induces is multiplication by $a$.  In particular, this gives an isomorphism in rational cohomology, and by the Five Lemma applied to the map of long exact sequences, we see that the cohomology of $R_{n,a-1}^{an,m}(\Cb)$ is as claimed.

    Finally, to see that \eqref{stabF} is an isomorphism for $d=an$ and $k=a-1$, we apply the Five Lemma to the map of long exact sequences induced by the continuous map of cofiber sequences
    \begin{equation*}
        \xymatrix{
            R_{n,a}^{an,m}(\Cb)^+ \ar[r] \ar[d] & R_{n,a-1}^{an,m}(\Cb)^+ \ar[r] \ar[d] & (\Poly_n^{n,m}(\Cb)\times\Cb^{a-1})^+ \ar[d]\\
            \ast \ar[r] & (R_{n,a-1}^{an-1,m}(\Cb)\times\Cb^m)^+ \ar[r] & (\Poly_n^{n-1,m}(\Cb)\times\Cb^{a-1}\times\Cb^m)^+
        }
    \end{equation*}
    The downward induction on $k$ now proceeds exactly as above, and this completes the proof of the claim.

    To conclude the inductive step of the theorem, we consider the map of cofiber sequences
    \begin{equation*}
        \xymatrix{
            R_{n,1}^{d,m}(\Cb)^+ \ar[r] \ar[d] & (\Cb^{dm})^+ \ar[r] \ar[d] & (\Poly_n^{d,m}(\Cb))^+ \ar[d]\\
            (R_{n,1}^{d-1,m}(\Cb)\times\Cb^m)^+ \ar[r] & (\Cb^{(d-1)m}\times\Cb^m)^+ \ar[r] & (\Poly_n^{d-1,m}(\Cb)\times\Cb^m)^+
        }
    \end{equation*}
    Applying the Five Lemma to the long exact sequence in cohomology, we see that the claim implies that the map \eqref{stabS} induces an isomorphism in compactly supported cohomology, and that
    \begin{equation}
    \label{eq:bettinums2}
        H^i_c(\Poly_n^{d,m}(\Cb);\Cb)\cong
            \left\{
                \begin{array}{ll}
                    \Cb & i=2(d-n)m+3\\
                    \Cb & i=2dm\\
                    0 & \text{else}
                \end{array}\right.
    \end{equation}

This establishes Statement (2) of the theorem.

\medskip
\noindent
{\bf Step 6 (Computing the weights): }
Statement (3) applied to \eqref{eq:bettinums2} gives that $H^i_{et,c}(\Poly_{n/_{\Fqbar}}^{d,m};\Qb_\ell)$ is one-dimensional for $i=2dm$ and $i=2(d-n)m+3$, and vanishes for all other
$i$.   Thus the trace of $\Frob_q$ on each of these cohomology groups is just the corresponding eigenvalue $\lambda_i$ of $\Frob_q$.  When $i=2dm$, Poincar\'{e} Duality implies that
$\lambda_{2dm}=q^{dm}$.   Plugging this information in to the Grothendieck-Lefschetz trace formula gives:

\[
\begin{array}{ll}
q^{dm}-q^{(d-n)m+1}=|\Poly_n^{d,m}(\F_q)|&=\lambda_{2dm}-\lambda_{2(d-n)m+3}\\
&\\
&=q^{dm}-\lambda_{2(d-n)m+3}
\end{array}
\]
which  implies that $\lambda_{2(d-n)m+3}=q^{(d-n)m+1}$, as claimed. This completes the proof of Statement (4) of the theorem.
\end{proof}

\section{The moduli space of  $m$-pointed rational curves in $\Pb^n$}
\label{section:M0m}

Fix $d,n\geq 1$ and $m\ge 0$.  For a variety $X$, let $\PConf_m(X)$ denote the set of ordered $m$-tuples of distinct points in $X$. Let

\begin{equation*}
    \Mc_{0,m}(\Pb^n,d):=(\PConf_m(\Pb^1)\times \Rat_d(\Pb^1,\Pb^n))/\PGL_2
\end{equation*}
where $\PGL_2$ acts diagonally. $\Mc_{0,m}(\Pb^n,d)$ is the {\em moduli space of $m$-pointed, degree $d$ rational curves in $\Pb^n$}. Note that the $m$ marked points on $\Pb^1$ give an {\em ordered} $m$-tuple.  Because $\PGL_2$ acts triply transitively on $\Pb^1$, if $m\ge 1$ then we can always take (a representative of) any element of $\Mc_{0,m}(\Pb^n,d)$ to have $\infty$ as the first marked point of the rational curve. There is then a map
\begin{equation*}
    \ev_\infty:\Mc_{0,m}(\Pb^n,d)\to \Pb^n
\end{equation*}
taking an equivalence class of marking and rational map $\phi:\Pb^1\to\Pb^n$ to $\phi(\infty)$.  It is straightforward to show that the map $\ev_\infty$ is a Zariski-locally trivial fibration with fiber
\begin{equation*}
    \Mc^*_{0,m}(\Pb^n,d):=(\PConf_{m-1}(\Ab^1)\times\Rat_d^*(\Pb^1,\Pb^n))/\Aff_1
\end{equation*}
where $\Aff_1$, the subgroup of $\PGL_2$ fixing $\infty\in\Pb^1$, acts diagonally. We thus have, for $m\ge 1$, and for any field $K$:
\begin{equation}
    \label{eq:M0mfibering1}
    [\Mc_{0,m}(\Pb^n,d)]=[\Mc^*_{0,m}(\Pb^n,d)]\cdot [\Pb^n] \in K_0(\Var_K)
\end{equation}
Our goal now is to compute the motive/point count, Betti numbers, and weight filtration of $\Mc^*_{0,m}(\Pb^n,d)$ for $m\ge 3$; we intend to return to the cases $m=1,2$ in a sequel.

\begin{proposition}\label{prop:iso1}
    For $m\ge 3$ there is an isomorphism
    \begin{equation*}
        \Mc^*_{0,m}(\Pb^n,d)\cong\PConf_{m-3}(\A^1-\{0,1\})\times\Rat_d^*(\Pb^1,\Pb^n).
    \end{equation*}
\end{proposition}
\begin{proof}
    For $m\ge 3$, we define a map
    \begin{equation*}
        \Psi: \PConf_{m-1}(\A^1)\times\Rat_d^*(\Pb^1,\Pb^n)\to \PConf_{m-3}(\A^1-\{0,1\})\times\Rat_d^*(\Pb^1,\Pb^n)
    \end{equation*}
    via
    \begin{equation*}
        \Psi((z_1,\ldots ,z_{m-1}),\phi):=(\beta(z_3),\ldots ,\beta(z_{m-1}),\phi\circ\beta^{-1})
    \end{equation*}
    where $\beta$ is the unique element of $\Aff_1$ so that $\beta(z_1)=0$ and $\beta(z_2)=1$. Note that we are making use of the simply transitive action of $\Aff_1$ on the space of pairs of distinct points in $\Ab^1$. Note also that for any $\alpha\in\Aff_1$:
    \begin{equation*}
        \Psi(\alpha\cdot (z_1,\ldots ,z_{m-1},\phi\circ\alpha^{-1})=\Psi((z_1,\ldots ,z_{m-1}),\phi)
    \end{equation*}
    and so $\Psi$ induces a map
    \begin{equation*}
        \overline{\Psi}: \Mc^*_{0,m}(\Pb^n,d)\to \PConf_{m-3}(\A^1-\{0,1\})\times\Rat_d^*(\Pb^1,\Pb^n).
    \end{equation*}
    The map $((z_3,\ldots,z_{m-1}),\phi)\mapsto ((0,1,z_3,\ldots,z_{m-1}),\phi)$ is an inverse to $\overline{\Psi}$, and so $\overline{\Psi}$ is an isomorphism.
\end{proof}

\subsection{Proof of Theorem~\ref{theorem:M0m}}

\begin{proof}[Proof of Theorem~\ref{theorem:M0m}]
    We begin with the motive/point count. Proposition \ref{prop:iso1} implies for $m\geq 3$ that:
    \begin{align*}
        [\Mc^\ast_{0,m}(\Pb^n,d)]=[\PConf_{m-3}(\Ab^1-\{0,1\})][\Ratb_{d,n}].
    \end{align*}

The following proposition was explained to us by N. Gadish.

\begin{proposition}[{\bf The class of \boldmath$\PConf_n(X)$}]
\label{prop:nir}
Let $X$ be a variety defined over $\Z$. Then for any field $K$:
    \begin{equation*}
        [\PConf_r(X)]=\prod_{i=0}^{r-1}([X]-i) \in K_0(\Var_K).
    \end{equation*}
\end{proposition}

\begin{proof} We prove this by induction on $r$. For $r=1$, there is nothing to show. Now assume we have shown it for $n<r$. Note
    that $\PConf_r(X)$ is the complement in $X^r$ of a union of diagonal subspaces isomorphic to $X^{r-1}$. All of the iterated intersections of these subspaces are again isomorphic to $X^i$ for $i<r$. By the inclusion--exclusion argument, this implies that $[\PConf_r(X)]$ is a polynomial in $[X]$ of degree $r$.

    The key observation is that this polynomial is independent of $X$: indeed, it depends only on the combinatorics of the natural stratification of the fat diagonal in $X^r$, i.e. only on the combinatorics of partitions of $\{1,\cdots,r\}$. But, this same argument gives that $|\PConf_r(X)(\F_q)|$ is the \emph{same} polynomial in $|X(\F_q)|$. Finally, by counting, we know that
    \begin{equation*}
        |\PConf_r(X)(\F_q)|=\prod_{i=0}^{r-1}(|X(\F_q)|-i).
    \end{equation*}
    This proves the proposition.
    \end{proof}

 Applying this formula with $X=\A^1-\{0,1\}$, so that $[X]=\Lb^1-2$, gives
    \begin{equation*}
        [\PConf_{m-3}(\A^1-\{0,1\})]=\prod_{i=2}^{m-2}(\Lb-i).
    \end{equation*}
    Combining this with the calculation for $[\Ratb_{d,n}]$, we obtain the result.

    For the Comparison Theorem, by Proposition \ref{prop:iso1}, the K\"unneth isomorphism, and Theorem \ref{thm:poly1}, we are reduced to producing an isomorphism
    \begin{equation*}
        H^\ast_{\et,c}(\PConf_{m-3}(\Ab^1-\{0,1\})_{/\Fqbar};\Qb_\ell)\otimes_{\Qb_\ell}\Cb\cong H^\ast_c(\PConf_{m-3}(\Cb-\{0,1\});\Cb).
    \end{equation*}
    Indeed, this follows from the same argument as for Theorem \ref{thm:poly1}: the variety $\PConf_{m-3}(\Ab^1-\{0,1\})$ is a complement of a hyperplane arrangement; the combinatorics of this arrangement are independent of the characteristic; and therefore, Bj\"orner--Ekedahl's results give the isomorphism.

    We now compute the weights. By Proposition \ref{prop:iso1}, Poinc\`are Duality and the K\"unneth isomorphism,
    \begin{align*}
        &H^i_{\et,c}(\Mc^\ast_{0,m}(\Pb^n,d)_{/\Fqbar};\Qb_\ell)\cong\\
        &\bigoplus_{a=0}^{2(m-3+d(n+1))-i}H^{2(m-3)-a}_{\et,c}(\PConf_{m-3}(\Ab^1-\{0,1\})_{/\Fqbar};\Qb_\ell)\otimes H^{i+a-2(m-3)}_{\et,c}(\Ratb_{d,n/\Fqbar};\Qb_\ell).
    \end{align*}
    The $\Ratb_{d,n}$ factor follows from Theorem \ref{thm:poly1}. It remains to compute the first factor. Recall that because $\PConf_{m-3}(\Ab^1-\{0,1\})$ is the complement of a hyperplane arrangement, its \'etale cohomology ring is generated in degree 1, and Frobenius acts on $H^i_{\et}(\PConf_{m-3}(\Ab^1-\{0,1\});\Qb_\ell)$ by $q^i$ \cite[Propositions 1, 2]{Ki}. Applying Poinc\`are duality, we see that Frobenius acts on $H^i_{\et}(\PConf_{m-3}(\Ab^1-\{0,1\});\Qb_\ell)$ by $q^{m-3-i}$.  It remains to compute the rank of the degree $i$ component, and by the Comparison Theorem, it suffices to do this using singular cohomology over $\Cb$.

    Ignoring the ring structure, the singular cohomology decomposes as a product
    \begin{equation*}
        H^\ast(\PConf_{m-3}(\Cb-\{0,1\});\Cb)\cong\bigotimes_{j=0}^{m-4}H^\ast(\Cb-\coprod_{j+2}\ast;\Cb).
    \end{equation*}
    Indeed, this follows from induction on $m$, and the Serre spectral sequence of the fibration
    \begin{equation*}
        \PConf_{m-3}(\Cb-\{0,1\})\to\PConf_{m-4}(\Cb-\{0,1\}).
    \end{equation*}
    The key observation is that because the fundamental group of the base acts trivially on the cohomology of the fiber, we have
    \begin{equation*}
        E_2^{p,q}=H^p(\PConf_{m-4}(\Cb-\{0,1\});\Cb)\otimes H^q(\Cb-\coprod_{m-4}\ast;\Cb),
    \end{equation*}
    and, because the fibration admits a section, the spectral sequence degenerates on this page.

    Using this decomposition, we see that the rank follows from the isomorphism of graded commutative rings
    \begin{equation*}
        H^\ast(\Cb-\coprod_i\ast;\Cb)\cong \Cb[x_1,\cdots,x_i]/<x_ax_b=0>
    \end{equation*}
    with $|x_a|=1$. In detail, we have
    \begin{align*}
        H^i_c(\PConf_{m-3}(\Cb-\{0,1\});\Cb)&\cong H^{2(m-3)-i}(\PConf_{m-3}(\Cb-\{0,1\});\Cb)\\
        &\cong\bigoplus_{i_0+\cdots+i_{m-4}=2(m-3)-i}\bigotimes_{j=0}^{m-4}H^{i_j}(\Cb-\coprod_{j+2}\ast;\Cb)\\
        &\cong\bigoplus_{\sigma}\bigotimes_{j=1}^{2(m-3)-i} H^1(\Cb-\coprod_{\sigma(j)+2}\ast;\Cb)\intertext{(where the sum is as in the definition of $\nu$ above)}
        &\cong\bigoplus_{\sigma}\Cb^{\prod_{j=1}^{2(m-3)-i}(\sigma(j)+2)}\\
        &\cong\Cb^{\sum_\sigma\prod_{j=1}^{2(m-3)-i}(\sigma(j)+2)}
    \end{align*}
    This completes the proof.
\end{proof}

\bigskip{\noindent
Dept. of Mathematics\\
University of Chicago\\
5734 S. University Avenue\\
Chicago, IL 60637\\
\\
E-mail:\\
farb@math.uchicago.edu\\

\bigskip{\noindent
Dept. of Mathematics\\
University of California - Irvine\\
Irvine, CA 92697\\
\\
E-mail:\\
wolfson@uci.edu
}

\end{document}